\documentclass[pdflatex,sn-mathphys-num]{sn-jnl}% Math and Physical Sciences Numbered Reference Style
\usepackage{multicol}
\usepackage{graphicx}%
\usepackage{multirow}%
\usepackage{amsmath,amssymb,amsfonts}%
\usepackage{amsthm}%
\usepackage{mathrsfs}%
\usepackage[title]{appendix}%
\usepackage{xcolor}%
\usepackage{textcomp}%
\usepackage{manyfoot}%
\usepackage{booktabs}%
\usepackage{algorithm}%
\usepackage{algorithmicx}%
\usepackage{algpseudocode}%
\usepackage{listings}%
\theoremstyle{thmstyleone}%
\newtheorem{theorem}{Theorem}%  meant for continuous numbers
\newtheorem{proposition}[theorem]{Proposition}%

\theoremstyle{thmstyletwo}%

\theoremstyle{thmstylethree}%
\newtheorem{definition}{Definition}%

\begin{document}

\title[S.D]{Geometric Dissipation Structures and the Koszul Form}

%%=============================================================%%
%% GivenName    -> \fnm{Joergen W.}
%% Particle -> \spfx{van der} -> surname prefix
%% FamilyName   -> \sur{Ploeg}
%% Suffix   -> \sfx{IV}
%% \author*[1,2]{\fnm{Joergen W.} \spfx{van der} \sur{Ploeg}
%%  \sfx{IV}}\email{iauthor@gmail.com}
%%=============================================================%%
\author*[1]{\fnm{} \sur{ Romain Nimpa Pefoukeu}}\email{romain.nimpa@facsciences-uy1.cm}
\author[1]{\fnm{} \sur{Prosper Rosaire Mama  Assandje }}\email{mamarosaire@facsciences-uy1.cm}\equalcont{These authors
contributed equally to this work.}
\author[1]{\fnm{} \sur{Michel Bertrand Djiadeu Ngaha }}\email{michel.djiadeu@facsciences-uy1.cm}\equalcont{These authors
contributed equally to this work.}
\author[1]{\fnm{} \sur{Frankel Nozah Mba }}\email{nozahmba@gmail.com}
\equalcont{These authors contributed equally to this work.}
\author[2]{\fnm{} \sur{Fr\'{e}d\'{e}ric Barbaresco }}\email{frederic.barbaresco@thalesgroup.com}
\equalcont{These authors contributed equally to this work.}

\affil*[1]{\orgdiv{Department of Mathematics}, \orgname{University
of Yaounde 1}, \orgaddress{\street{usrectorat@.univ-yaounde1.cm},
\city{yaounde}, \postcode{812}, \state{Center}, \country{Cameroon}}}
 \affil[2]{\orgdiv{Department of Mathematics}, \orgname{Thales Land \& Air Systems(
Industry ) }, \orgaddress{\street{https://www.thalesgroup.com/en},
\city{2 Avenue Gay Lussac}, \postcode{ CS90502 78990 Elancourt},
\state{Region 8 (Africa, Europe, Middle East)}, \country{France}}}

%%==================================%%
%% Sample for unstructured abstract %%
%%==================================%%

\abstract{This paper investigates the canonical Koszul one-form at
the interface of Jean Marie Souriau's Lie group thermodynamics and
Paulette Libermann's symplectically complete foliations to describe
geometric dissipation. On the Poincar\'{e} half-plane $\mathbb{H}$
under $SL(2, \mathbb{R})$, the equivariant momentum map level sets
form a regular Libermann foliation framing the conservative
boundaries. The Koszul-derived dissipative vector field acts
strictly transverse to these leaves, driving a metriplectic flow
that breaks Noether conservation laws and generates entropy. Quantum
wise, we develop Berezin quantization on para-K\"{a}hler symmetric
spaces, where the Koszul potential acts as a weight determining
state density via the Fisher Souriau metric, and formalize Kostant's
geometric quantization using the real dual polarizations of a
bi-Lagrangian web.}

\keywords{Dissipation Structures, Libermann foliation, momentum map,
metriplectic flow}

%%\pacs[JEL Classification]{D8, H51}

\pacs[MSC Classification]{53D20, 53D50, 53C12, 82C05, 53B12}

\maketitle

\section{Introduction}\label{sec1}

The geometrization of classical mechanics and thermodynamics
constitutes one of the major pillars of modern mathematical physics.
Initiated by the pioneering work of Jean Marie Souriau
\cite{souriau1970}, Lie group thermodynamics provides a rigorous
mathematical framework where statistical equilibrium states are
interpreted as probability measures, specifically Gibbs ensembles,
supported on the coadjoint orbits of a Lie group $G$. Within this
variational formalism, the macroscopic properties of a physical
system are intrinsically encoded by the natural
Kirillov-Kostant-Souriau (KKS) symplectic geometry carried by these
orbits, establishing a deep-rooted connection between algebraic
symmetries and macroscopic thermal states \cite{marle2020,
desaxce2026}. The study of these orbits carrying Gibbs ensembles has
recently culminated in explicit classification theorems, confirming
the profound impact of such configurations on modern information
geometry and strongly Hamiltonian actions \cite{neeb2026}. However,
the geometric description of non-equilibrium phenomena and
irreversible dissipation structures requires the introduction of
transverse topological and differential machinery capable of
reaching beyond the strictly conservative boundaries of traditional
symplectic systems. Recent research has focused heavily on modeling
irreversible dynamics directly on Poisson manifolds by leveraging
structural deformation theories \cite{luesink2026} rooted in the
local foundations of Alan Weinstein \cite{Weinstein1983} and
Andr\'{e} Lichnerowicz \cite{Lichnerowicz1977}, or by extending
entropy functionals to infinite-dimensional Hamiltonian frameworks
\cite{magnot2026}. In parallel, the symplectic foliation formulation
applied to Sadi Carnot's foundational thermodynamics
\cite{barbaresco2025} has shed light on the structural necessity of
a dual geometric infrastructure designed to isolate reversible flows
from orthogonal dissipative forces. In \cite{mama2026}, we study the
generalized Fisher metric on the Lie groups $SO(2)$ and $SO(3)$ via
Souriau's Lie group thermodynamics. Then we give the effect of a
two-cocycle on the integrability of gradient systems due to the
Fisher metric and the Souriau-Fisher metric. In addition, we show
how the cocycle can locally modify the Fisher metric on a coadjoint
orbit, in explicit terms of brackets and central extensions on the
Lie groups $SO(2)$ and $SO(3)$. The present work positions itself at
the precise confluence of these mathematical disciplines.
Specifically, we investigate the following core problem: What role
does the Koszul one-form play between the symplectic foliation
structures transversal to topics of Jean-Marie Souriau's Lie group
thermodynamics and the bi-Lagrangian or two-Lagrangian web
structures and Libermann foliations (K\"{u}nneth structure) to
describe geometric dissipation structures? By replacing the
classical unit disk framework with the intrinsic geometry of the
Poincar\'{e} half-plane $\mathbb{H}$ under the strongly Hamiltonian
action of $SL(2, \mathbb{R})$, we show that the level sets of the
equivariant momentum map $J(z,\bar{z})$ define a regular Libermann
foliation representing the conservative boundaries of the system.
The underlying classical structure builds upon information geometry
as originally developed by Amari \cite{Shu} and revisited through
the elementary frameworks of Jean Louis Koszul \cite{barbaresco}.
The canonical Koszul one-form and its associated Hessian potential
act as fundamental thermodynamic mediators across phase space
\cite{souriau, barbaresco1}. In the presence of a bi-foliation or a
two-Lagrangian web structure, the tangent bundle splits naturally
into complementary subbundles. This geometric reduction closely
mirrors the regular and singular foliation mechanics and equivalence
problems studied historically by Haefliger \cite{Haefliger1958},
Fedida \cite{Fedida1971}, and Paulette Libermann
\cite{Libermann1983-1, Libermann1983-2, LibermannMarle1986}, as well
as the overarching global variables and topology of integrable
Hamiltonian systems with commuting integrals analyzed by Duistermaat
\cite{Duistermaat1980}. It also connects profoundly with the
convexity properties of the momentum mapping established by Atiyah
\cite{Atiyah1982}, Guillemin and Sternberg
\cite{GuilleminSternberg1982, GuilleminSternberg1984}, and Kirwan
\cite{Kirwan1984}, alongside Eliasson's local normal forms
\cite{Eliasson1984} and the general action-angle framework of Dazord
and Delzant \cite{DazordDelzant1987} or Dufour and Molino
\cite{DufourMolino1987}. The integrability of these distributions,
firmly anchored in the mathematical work on singular foliations and
field orbits by Stefan \cite{Stefan1974} and Sussmann
\cite{Sussmann1973}, provides the ideal geometric substrate where
the non-equilibrium Koszul flow drives the system transversally and
orthogonally across the conservative level sets of the momentum map.
To bridge the gap with the microscopic realm, this manuscript also
addresses the explicit extension of this geometric architecture to
the quantum domain. First, we investigate Berezin quantization on
the para-K\"{a}hler symmetric spaces of a semi-simple Lie group,
where the para-complex signature induces a split geometry that
perfectly aligns with the K\"{u}nneth decomposition of Libermann
foliations. Second, we analyze Kostant quantization on a general
symplectic manifold endowed with a bi-Lagrangian web. We establish
that the duality of the Lagrangian leaves provides a natural pair of
real complementary polarizations. These polarizations allow the
definition and coupling of quantum state spaces through a modified
prequantum connection and a Blattner-Kostant-Sternberg (BKS)
transition kernel explicitly scaled by the path of thermodynamic
Koszul dissipation. The outline of this paper follows this geometric
progression. Section \ref{sec2} establishes the algebraic
preliminaries and fundamental notions regarding coadjoint orbits.
Section \ref{sec3} details the structural mapping of the momentum
map, the Koszul potential, the stable open dual cone $\Omega^*$, and
the Fisher-Souriau information metric tensor. Section \ref{sec4}
formulates the metriplectic flow kinetics driving non-equilibrium
trajectories. Section \ref{sec5} formalizes our main structural
theorem regarding the Stefan-Libermann foliation, its exact
point-like sheet geometry $J^{-1}(\xi)$, and its strict
transversality against the dissipative dynamics. Finally, Sections
\ref{sec6} and \ref{sec7} develop the respective quantum extensions
under the constraints of Berezin and Kostant bi-Lagrangian
quantization, followed by a discussion in Section \ref{sec8} and a
synthetic conclusion in Section \ref{sec9}.

\section{Preliminaries}\label{sec2}

\subsection{Lie Algebra}\label{subsec1}
We can define Lie algebras over any field. We restrict ourselves to
the real or complex case. So let $\mathbb{K}=\mathbb{R}$ or
$\mathbb{C}$.

\begin{definition}\textup{\cite{Lie}}\label{def:1}
A Lie algebra $\mathfrak{a}$ over $\mathbb{K}$ is a finite- or
infinite-dimensional $\mathbb{K}$-vector space equipped with a
$\mathbb{K}$-bilinear, antisymmetric operation $[\cdot, \cdot]$
satisfying the Jacobi identity:
\begin{equation}\label{eq:jacobi}
    X, Y, Z\in\mathfrak{a},\quad [X,[Y,Z]] + [Y,[Z, X]] + [Z,[X, Y]] = 0.
\end{equation}
The bilinear operation $[\cdot, \cdot]$ is called the Lie bracket
(or simply bracket). The center of a Lie algebra $\mathfrak{a}$ is
the abelian ideal of $\mathfrak{a}$ defined by
\begin{equation}\label{eq:center}
    \left\{ X \in \mathfrak{a} \, ; \, \forall Y \in \mathfrak{a}, \quad [X,Y] = 0 \right\}.
\end{equation}
A Lie algebra is called simple if it has no nontrivial ideals and is
not of dimension $0$ or $1$. It is called semi-simple if it has no
nonzero abelian ideals.
\end{definition}

\begin{definition}\textup{\cite{Lie}}\label{def:2}
A Lie subalgebra of a Lie algebra is a vector subspace closed under
the Lie bracket.
\end{definition}

\begin{proposition}\textup{\cite{Lie}}\label{prop:1}
The Lie algebras of $\mathrm{SL}(n,\mathbb{K})$, $\mathrm{O}(n)$,
and $\mathrm{SO}(n)$ are given by:
\begin{align*}
  \mathfrak{sl}(n,\mathbb{K}) &= \left\{X\in \mathfrak{gl}(n,\mathbb{K})\, ; \,\mathrm{Tr}(X)=0\right\},
\end{align*}
which is the Lie algebra of traceless $n\times n$ matrices with
entries in $\mathbb{K}$, of dimensions
$\dim_{\mathbb{R}}\mathfrak{sl}(n,\mathbb{R}) = n^{2}-1$,
$\dim_{\mathbb{C}}\mathfrak{sl}(n,\mathbb{C}) = n^{2}-1$, and
$\dim_{\mathbb{R}}\mathfrak{sl}(n,\mathbb{C}) = 2(n^{2}-1)$; and
\begin{align*}
  \mathfrak{o}(n) &= \mathfrak{so}(n) = \left\{X\in \mathfrak{gl}(n,\mathbb{R}) \, ; \, X + {}^{t}\!X = 0\right\},
\end{align*}
which is the Lie algebra of skew-symmetric real $n\times n$
matrices, of dimension $\dim \mathfrak{so}(n) = n(n-1)/2$.
\end{proposition}

\subsection{Coadjoint Operator and Coadjoint Orbits (Kirillov Representation)}\label{subsec2}

\begin{definition}\textup{\cite{souriau,souriau1}}\label{def:3}
The adjoint representation of a Lie group $\mathrm{Ad}_{g}$ is a way
of representing its elements as linear transformations of the Lie
algebra, considered as a vector space:
\begin{align*}
  \varphi: G &\longrightarrow \mathrm{Aut}(G), \quad g \longmapsto \varphi_{g}(h) = ghg^{-1}, \\
  \mathrm{Ad}_{g} = T_{e}\varphi_{g}: \mathfrak{g} &\longrightarrow \mathfrak{g}, \quad X \longmapsto \mathrm{Ad}_{g}(X) = gXg^{-1}, \\
  \mathrm{ad} = T_{e}\mathrm{Ad}: T_{e}(G) &\longrightarrow \mathrm{End}\left(T_{e}(G)\right), \quad (X,Y) \longmapsto \mathrm{ad}_{X}(Y) = \left[X,Y\right].
\end{align*}
The coadjoint representation of a Lie group $\mathrm{Ad}^{*}_{g}$ is
the dual of the adjoint representation (where $\mathfrak{g}^{*}$
denotes the dual space to $\mathfrak{g}$):
\begin{equation*}
    \forall g\in G,\; Y\in \mathfrak{g},\; F\in \mathfrak{g}^{*}, \quad \langle \mathrm{Ad}^{*}_{g}F, Y \rangle = \langle F, \mathrm{Ad}_{g^{-1}}Y \rangle.
\end{equation*}
A coadjoint orbit
\begin{equation*}
    O_{F} = \left\{ \mathrm{Ad}^{*}_{g}F = g^{-1}Fg, \; g\in G, \; F\in \mathfrak{g}^{*} \right\}, \quad K = \mathrm{Ad}^{*}_{g} = \left(\mathrm{Ad}_{g^{-1}}\right)^{*}, \quad \text{and} \quad K_{*}X = -\left(\mathrm{ad}_{X}\right)^{*}
\end{equation*}
carries a natural homogeneous symplectic structure defined by a
closed $G$-invariant two-form:
\begin{equation*}
    \sigma_{\Omega}\left(K_{*X}F, K_{*Y}F\right) = B_{F}\left(X,Y\right) = \langle F, \left[X,Y\right] \rangle, \qquad X,Y \in \mathfrak{g}.
\end{equation*}
The coadjoint action on $\Omega$ is a Hamiltonian $G$-action with
moment map given by the inclusion $\Omega \longrightarrow
\mathfrak{g}^{*}$.
\end{definition}

\begin{definition}\textup{\cite{souriau,souriau1}}\label{def:4}
The tensor $\tilde{\Theta}$ is defined in the tangent space of the
cocycle $\theta(g)\in\mathfrak{g}^{*}$ (this cocycle appears due to
the non-equivariance of the coadjoint operator
$\mathrm{Ad}^{*}_{g}$, the action of the group on the dual Lie
algebra): $\eta\left(\mathrm{Ad}_{g}(\beta)\right) =
\mathrm{Ad}^{*}_{g}\left(\eta(\beta)\right) + \theta(g)$,
\begin{equation*}
  \tilde{\Theta}: \mathfrak{g}\times \mathfrak{g} \longrightarrow \mathbb{R}, \quad (X,Y) \longmapsto \langle\Theta(X), Y\rangle,
\end{equation*}
with $\Theta(X) = T_{e}\theta(X(e))$. According to Souriau
\textup{\cite{souriau,souriau1}}, the generalized information metric
is given by $I(\beta) = \left[g_{\beta}\right]$, where
\begin{equation*}
  g_{\beta}\left(\left[\beta,Z_{1}\right], \left[\beta,Z_{2}\right]\right) = \tilde{\Theta}_{\beta}\left(Z_{1}, \left[\beta,Z_{2}\right]\right),
\end{equation*}
with $\tilde{\Theta}_{\beta}\left(Z_{1},Z_{2}\right) =
\tilde{\Theta}\left(Z_{1},Z_{2}\right) + \langle \eta,
\mathrm{ad}_{Z_{1}}(Z_{2}) \rangle$.
\end{definition}

\subsection{Foliations}\label{subsec3}

\begin{definition}\textup{\cite{feuille}}\label{def:5}
A foliation of dimension $n$ of a differentiable manifold $M^{m}$
is, roughly speaking, a decomposition of $M$ into connected
submanifolds of dimension $n$ called leaves, which locally stack up
like the subsets of $\mathbb{R}^{m}=\mathbb{R}^{n}\times
\mathbb{R}^{m-n}$ with the second coordinate constant. The
diffeomorphisms
\begin{equation*}
h: U \subset \mathbb{R}^{m} \longrightarrow V \subset \mathbb{R}^{m}
\end{equation*}
which preserve the leaves of this foliation locally have the
following form:
\begin{equation}\label{00}
h(x,y)=(h_{1}(x,y),h_{2}(y)),\quad (x,y)\in \mathbb{R}^{n}\times
\mathbb{R}^{m-n}.
\end{equation}
\end{definition}

\begin{definition}\textup{\cite{feuille}}\label{def:6}
Let $M$ be a $C^{\infty}$ manifold of dimension $m$. A $C^{r}$
foliation of dimension $n$ of $M$ is a $C^{r}$ atlas $\mathcal{A}$
on $M$ which is maximal with the following properties:
\begin{enumerate}
    \item If $(U, \varphi)\in\mathcal{A}$, then $\varphi(U) = U_{1} \times U_{2}\subset
      \mathbb{R}^{n}\times \mathbb{R}^{m-n}$ where $U_{1}$ and $U_{2}$ are
      open disks in $\mathbb{R}^{n}$ and $\mathbb{R}^{m-n}$ respectively.
    \item If $(U, \varphi)$ and $(V, \psi)\in\mathcal{A}$ are such that $U\cap V\neq \emptyset$,
      then the change of coordinates map
      $\psi\circ \varphi^{-1} : \varphi(U\cap V)\longrightarrow \psi(U\cap V)$ is of the form
      (\ref{00}), that is, $\psi\circ \varphi^{-1}(x,y) = (h_{1}(x,y),h_{2}(y))$. We say that $M$ is foliated by $\mathcal{A}$, or that $\mathcal{A}$ is a
      foliated structure of dimension $n$ and class $C^{r}$ on $M$.
\end{enumerate}
\end{definition}

\begin{definition}\textup{\cite{feuille}}\label{def:7}
A $C^{\infty}$ action of a Lie group $G$ on a manifold $M$ is a map
$\psi_{1}: G \times M \rightarrow M$ such that $\psi_{1}
\left(e,x\right) = x$ and $\psi_{1}\left(g_{1} g_{2}, x\right) =
\psi_{1}\left(g_{1}, \psi_{1}\left(g_{2}, x\right)\right)$ for any
$g_{1},g_{2}\in G$ and $x \in M$. The orbit of a point $x \in M$ for
the action $\psi_{1}$ is the subset
$\mathcal{O}_{x}(\psi_{1})=\left\{\psi_{1}\left(g,x\right)\in M \mid
g\in G\right\}$. The isotropy group of $x \in M$ is the subgroup
$\mathcal{G}_{x}(\psi_{1})=\left\{g\in G \mid \psi_{1}
\left(g,x\right)=x\right\}$.
\end{definition}

\begin{proposition}\textup{\cite{feuille}}\label{prop:2}
The orbits of a foliated action define the leaves of a foliation.
\end{proposition}

\subsection{Hessian Structures and the Koszul One-Form}\label{subsec4}

Following the information-geometric formulations of Amari \cite{Shu}
and the foundational lectures of Koszul \cite{barbaresco}, a Hessian
structure on a manifold or a stable open cone $\Omega \subset
\mathfrak{g}$ provides the native framework connecting thermodynamic
potentials to metric geometry.

\begin{definition}\textup{\cite{barbaresco, souriau}}\label{def:8}
Let $\Omega$ be a flat domain equipped with a flat affine connection
$D$. A Riemannian metric $g$ on $\Omega$ is called a Hessian metric
if it can be locally expressed as the Hessian of a smooth convex
function $\Phi$, known as the thermodynamic or characteristic
potential:
\begin{equation}
g_{ij} = \frac{\partial^2 \Phi}{\partial \beta_i \partial \beta_j}.
\end{equation}
The canonical Koszul one-form $\beta$ associated with this geometry
is defined as the differential of the potential, $\beta =
\mathrm{d}\Phi$, which acts as an entry point for defining the
information volume element.
\end{definition}

In Jean-Marie Souriau's Lie group thermodynamics \cite{souriau1970},
when the coadjoint operator action is non-equivariant, the presence
of the Souriau symplectic cocycle $\theta(g)$ modifies the
transformation laws on $\mathfrak{g}^*$ \cite{souriau}. Under this
setting, the Hessian matrix of the potential $\Phi(\beta) = -\ln
\chi(\beta)$ perfectly mirrors the components of the Fisher--Souriau
information metric $I(\beta)$, establishing a bridge between
statistical mechanical states and the underlying homogeneous
symplectic architecture \cite{barbaresco1, marle2020}.

\subsection{Bi-Lagrangian Webs and Para-K\"{a}hler Manifolds}\label{subsec5}
To accommodate geometric dissipation within a conservative
symplectic universe, we must invoke the dual geometric structures
introduced by Paulette Libermann \cite{libermann1983ref,
libermann1983ast}.

\begin{definition}\textup{\cite{libermann1983ref}}\label{def:9}
A symplectic manifold $(M, \omega)$ of dimension $2n$ is said to
admit a bi-Lagrangian structure (or a 2-Lagrangian web) if its
tangent bundle splits into a direct sum of two complementary
Lagrangian distributions:
\begin{equation}
TM = T\mathcal{F}_1 \oplus T\mathcal{F}_2,
\end{equation}
where $\omega(X, Y) = 0$ for all $X, Y \in T\mathcal{F}_i$
($i=1,2$). The pair $(\mathcal{F}_1, \mathcal{F}_2)$ defines a
symplectically complete foliation structure, or a Libermann
foliation, whose coordinates satisfy a split K\"{u}nneth
decomposition topology.
\end{definition}

\subsection{Hamiltonian Actions and the Momentum Map}\label{subsec6}

Let $(M, \omega)$ be a smooth symplectic manifold. A vector field
$X$ on $M$ is called Hamiltonian if there exists a smooth function
$f \in C^\infty(M)$ such that $\iota_X \omega = \mathrm{d}f$. To
lift a symmetry group action to the symplectic framework, we
introduce the formal definition of a momentum map.

\begin{definition}\textup{\cite{LibermannMarle1987}}\label{def:10}
Let $G$ be a Lie group acting on $(M, \omega)$ by
symplectomorphisms, and let $\mathfrak{g}$ be its associated Lie
algebra. The action is called Hamiltonian if there exists a smooth
map $J : M \to \mathfrak{g}^*$, called the momentum map (or
\textit{application moment}), such that:
\begin{enumerate}
    \item For every element $X \in \mathfrak{g}$, let $X_M$ be the fundamental vector field generated on $M$ by the infinitesimal action. Then:
    \begin{equation}
    \iota_{X_M} \omega = \mathrm{d}\langle J, X \rangle,
    \end{equation}
    where $\langle \cdot, \cdot \rangle$ denotes the natural duality pairing between $\mathfrak{g}^*$ and $\mathfrak{g}$.
    \item The map $J$ is equivariant with respect to the action of $G$ on $M$ and the coadjoint action $\mathrm{Ad}^*$ of $G$ on $\mathfrak{g}^*$, meaning $J(g \cdot x) = \mathrm{Ad}_g^*(J(x))$ for all $g \in G, x \in M$.
\end{enumerate}
\end{definition}

As detailed in Souriau's core formulation \cite{souriau1970}, when
the coadjoint action is non-equivariant due to topological
constraints, a symplectic cocycle $\theta(g)$ appears, altering the
relation to $J(g \cdot x) = \mathrm{Ad}_g^*(J(x)) + \theta(g)$,
which directly leads to the definition of the generalized
information metric tensor $\tilde{\Theta}$ via the affine
transformation properties on the stable open cone $\Omega$.

\begin{definition}\textup{\cite{libermann1983ast, LibermannMarle1986}}\label{def:11}
A foliation $\mathcal{F}$ on a symplectic manifold $(M, \omega)$ is
called a Libermann foliation (or a symplectically complete
foliation) if the space of smooth functions that are constant along
the leaves of $\mathcal{F}$ is closed under the Poisson bracket
induced by $\omega$. That is:
\begin{equation}
\{f, g\} \in C^\infty(M / \mathcal{F}) \quad \forall f, g \in
C^\infty(M / \mathcal{F}),
\end{equation}
where $C^\infty(M / \mathcal{F})$ denotes the algebra of first
integrals of the foliation.
\end{definition}

\subsection{The Poincar\'{e} Half-Plane Model $\mathbb{H}$ and the Poincar\'{e} Unit Disk Model $\mathbb{D}$}\label{subsec7}

\begin{definition}\textup{\cite{souriau1970}}\label{def:12}
The Poincar\'{e} half-plane, denoted by $\mathbb{H}$, is the
Riemannian manifold defined by the open upper half of the complex
plane:
\begin{equation}
\mathbb{H} = \left\{ z \in \mathbb{C} \ \middle|\
\operatorname{Im}(z) > 0 \right\},
\end{equation}
equipped with the conformal metric (Poincar\'{e} metric):
\begin{equation}
ds^2_{\mathbb{H}} = \frac{dx^2 + dy^2}{y^2},
\end{equation}
where $z = x + iy$. The hyperbolic geodesics within this space
consist of Euclidean vertical half-lines and Euclidean semicircles
centered on and orthogonal to the real axis.
\end{definition}

\begin{definition}\textup{\cite{LibermannMarle1986}}\label{def:13}
The Poincar\'{e} unit disk, denoted by $\mathbb{D}$, is the
Riemannian manifold defined by the open unit disk in the complex
plane:
\begin{equation}
\mathbb{D} = \left\{ w \in \mathbb{C} \ \middle|\ |w| < 1 \right\},
\end{equation}
equipped with the conformal metric:
\begin{equation}
ds^2_{\mathbb{D}} = 4\frac{du^2 + dv^2}{(1 - |w|^2)^2},
\end{equation}
where $w = u + iv$. The hyperbolic geodesics consist of the disk's
diameters and arcs of Euclidean circles intersecting the unit
boundary circle $|w|=1$ at strict right angles.
\end{definition}

\section{Moment Map, Fisher Metric, and Shannon Entropy}\label{sec3}

\begin{theorem}\label{thm:1}
Let $\left(\mathbb{H},\omega\right)$ be a homogeneous manifold with
$\omega=\frac{-2i}{\left(z-\bar{z}\right)^{2}}dz\wedge d\bar{z}$ the
closed K\"{a}hler two-form on the Poincar\'{e} half-plane. For any
$\alpha_{i}\in \mathfrak{sl}(2,\mathbb{R})$, $i\in \{1,2,3\}$, the
action of the group $SL(2,\mathbb{R})$ given by
$\phi:SL(2,\mathbb{R})\times \mathbb{H}\rightarrow \mathbb{H},\quad
z\mapsto g\cdot z=\frac{az+b}{cz+d}$ generates the fundamental
vector fields $X_{1}(z)=1+z^2$, $X_{2}(z)=2z$, and
$X_{3}(z)=1-z^{2}$ associated respectively with the first integrals
$J_{1}(z,\bar{z})=\frac{-2i\left(1+z\bar{z}\right)}{z-\bar{z}}$,
$J_{2}(z,\bar{z})=\frac{-2i\left(z+\bar{z}\right)}{z-\bar{z}}$, and
$J_{3}(z,\bar{z})=\frac{-2i\left(1-z\bar{z}\right)}{z-\bar{z}}$
while leaving $\omega$ invariant. There exists an equivariant
momentum map $J:\mathbb{H}\rightarrow
\mathfrak{sl}^{*}(2,\mathbb{R})$ defined by its matrix form:
\begin{equation*}
J(z,\bar{z})=\frac{2i}{z-\bar{z}}\begin{pmatrix}
-(z+\bar{z}) & 2z\bar{z} \\
-2  & z+\bar{z}
\end{pmatrix}.
\end{equation*}
\end{theorem}

\begin{proof}
Let
\begin{equation*}
\alpha_1 = \begin{pmatrix} 0 & 1 \\ -1 & 0 \end{pmatrix}, \quad
\alpha_2 = \begin{pmatrix} 1 & 0 \\ 0 & -1 \end{pmatrix}, \quad
\alpha_3 = \begin{pmatrix} 0 & 1 \\ 1 & 0 \end{pmatrix}
\end{equation*}
be a basis of the Lie algebra $\mathfrak{sl}(2,\mathbb{R})$. For any
$\alpha_{i}\in \mathfrak{sl}(2,\mathbb{R})$, $i\in \{1,2,3\}$, the
infinitesimal action is given by
$X_{i}(z)=\left.\frac{d}{dt}\right|_{t=0}\exp(t\alpha_{i})\cdot z$,
for $i\in \{1,2,3\}$. The fundamental vector fields are given by
\begin{align*}
  X_{1}(z) &= \left.\frac{d}{dt}\right|_{t=0}\exp(t\alpha_{1})\cdot z, \quad
  X_{2}(z) = \left.\frac{d}{dt}\right|_{t=0}\exp(t\alpha_{2})\cdot z, \\
  X_{3}(z) &= \left.\frac{d}{dt}\right|_{t=0}\exp(t\alpha_{3})\cdot z,
\end{align*}
where
\begin{equation*}
\exp(t\alpha_{1}) = \begin{pmatrix}
  \cos t & \sin t \\
  -\sin t & \cos t
\end{pmatrix}, \quad
\exp(t\alpha_{2}) = \begin{pmatrix}
  \exp(t) & 0 \\
  0 & \exp(-t)
\end{pmatrix}, \quad
\exp(t\alpha_{3}) = \begin{pmatrix}
  \cosh(t) & \sinh(t) \\
  \sinh(t) & \cosh(t)
\end{pmatrix}.
\end{equation*}
The action of the group $SL(2,\mathbb{R})$ given by
$\phi:SL(2,\mathbb{R})\times \mathbb{H}\rightarrow \mathbb{H},\quad
g\cdot z=\frac{az+b}{cz+d}$ generates the fundamental vector fields
$X_{1}(z)=1+z^2$, $X_{2}(z)=2z$, and $X_{3}(z)=1-z^{2}$. These
fundamental fields are Hamiltonian vector fields, i.e.,
$\iota_{X_{k}}\omega=\mathrm{d}J_{k}$, $k=1,2,3$.

Thus, for each field we have
\begin{align*}
  \iota_{X_{1}}\omega &= \frac{-2i}{\left(z-\bar{z}\right)^{2}}\left(\iota_{X_{1}}dz\wedge d\bar{z}\right), \quad
  \iota_{X_{2}}\omega = \frac{-2i}{\left(z-\bar{z}\right)^{2}}\left(\iota_{X_{2}}dz\wedge d\bar{z}\right), \\
  \iota_{X_{3}}\omega &= \frac{-2i}{\left(z-\bar{z}\right)^{2}}\left(\iota_{X_{3}}dz\wedge d\bar{z}\right).
\end{align*}
Therefore, by the Leibniz rule, we have
\begin{align}\label{e}
  \iota_{X_{1}}\omega &= \frac{-2i}{\left(z-\bar{z}\right)^{2}}\left((\iota_{X_{1}}dz)d\bar{z}- dz(\iota_{X_{1}}d\bar{z})\right), \notag \\
  \iota_{X_{2}}\omega &= \frac{-2i}{\left(z-\bar{z}\right)^{2}}\left((\iota_{X_{2}}dz)d\bar{z}- dz(\iota_{X_{2}}d\bar{z})\right), \notag \\
  \iota_{X_{3}}\omega &= \frac{-2i}{\left(z-\bar{z}\right)^{2}}\left((\iota_{X_{3}}dz)d\bar{z}-dz(\iota_{X_{3}}d\bar{z})\right),
\end{align}
where
\begin{align}\label{e1}
  X_{1}(z) &= (1+z^2)\frac{\partial}{\partial z}+(1+\bar{z}^2)\frac{\partial}{\partial \bar{z}}, \quad
  X_{2}(z) = 2z\frac{\partial}{\partial z}+2\bar{z}\frac{\partial}{\partial \bar{z}}, \notag \\
  X_{3}(z) &= (1-z^{2})\frac{\partial}{\partial z}+(1-\bar{z}^{2})\frac{\partial}{\partial \bar{z}}.
\end{align}
Using \eqref{e1} in \eqref{e}, we obtain
\begin{align}\label{e2}
  \iota_{X_{1}}\omega &= \frac{-2i}{\left(z-\bar{z}\right)^{2}}\left((1+z^2)d\bar{z}-(1+\bar{z}^2)dz\right), \quad
  \iota_{X_{2}}\omega = \frac{-2i}{\left(z-\bar{z}\right)^{2}}\left(2zd\bar{z}-2\bar{z}dz\right), \notag \\
  \iota_{X_{3}}\omega &= \frac{-2i}{\left(z-\bar{z}\right)^{2}}\left((1-z^{2})d\bar{z}-(1-\bar{z}^{2})dz\right),
\end{align}
where
\begin{align}\label{e3}
  dJ_{1} &= \frac{\partial J_{1}}{\partial z}dz+\frac{\partial J_{1}}{\partial \bar{z}}d\bar{z}, \quad
  dJ_{2} = \frac{\partial J_{2}}{\partial z}dz+\frac{\partial J_{2}}{\partial \bar{z}}d\bar{z}, \notag \\
  dJ_{3} &= \frac{\partial J_{3}}{\partial z}dz+\frac{\partial J_{3}}{\partial \bar{z}}d\bar{z}.
\end{align}
By identification, we obtain the following differential systems:
\begin{multicols}{3}
\noindent $\begin{cases}
  \dfrac{\partial J_{1}}{\partial z} = \dfrac{2i(1+\bar{z}^2)}{\left(z-\bar{z}\right)^{2}} \\[1.2ex]
  \dfrac{\partial J_{1}}{\partial \bar{z}} = \dfrac{-2i(1+z^2)}{\left(z-\bar{z}\right)^{2}},
\end{cases}$

\noindent $\begin{cases}
  \dfrac{\partial J_{2}}{\partial z} = \dfrac{4i\bar{z}}{\left(z-\bar{z}\right)^{2}} \\[1.2ex]
  \dfrac{\partial J_{2}}{\partial \bar{z}} = \dfrac{-4iz}{\left(z-\bar{z}\right)^{2}},
\end{cases}$

\noindent $\begin{cases}
  \dfrac{\partial J_{3}}{\partial z} = \dfrac{2i(1-\bar{z}^2)}{\left(z-\bar{z}\right)^{2}} \\[1.2ex]
  \dfrac{\partial J_{3}}{\partial \bar{z}} = \dfrac{-2i(1-z^2)}{\left(z-\bar{z}\right)^{2}}.
\end{cases}$
\end{multicols}
Integrating these systems yields:
\begin{equation*}
J_{1}(z,\bar{z})=\frac{-2i\left(1+z\bar{z}\right)}{z-\bar{z}},\quad
J_{2}(z,\bar{z})=\frac{-2i\left(z+\bar{z}\right)}{z-\bar{z}},\quad
J_{3}(z,\bar{z})=\frac{-2i\left(1-z\bar{z}\right)}{z-\bar{z}}.
\end{equation*}
A dual basis is given by
\begin{equation*}
\alpha^{*}_1 = \begin{pmatrix} 0 & -1 \\ 1 & 0 \end{pmatrix}, \quad
\alpha^{*}_2 = \begin{pmatrix} 1 & 0 \\ 0 & -1 \end{pmatrix}, \quad
\alpha^{*}_3 = \begin{pmatrix} 0 & 1 \\ 1 & 0 \end{pmatrix}.
\end{equation*}
The duality coupling is defined by
$\langle\alpha,\alpha^{*}\rangle:=\frac{1}{2}\mathrm{Tr}(\alpha\alpha^{*})$.
\end{proof}

\begin{theorem}\label{tho1}
Let $\mathfrak{sl}(2,\mathbb{R})$ be the three-dimensional real
traceless Lie algebra equipped with the basis
$\{\alpha_{1},\alpha_{2},\alpha_{3}\}$ and the stable open cone
$\Omega = \left\{\beta\in\mathfrak{sl}(2,\mathbb{R}) \mid
\beta^{2}_{1}-\beta^{2}_{2}-\beta^{2}_{3}>0\right\}$. The
thermodynamic potential $\Phi(\beta) = -\ln \chi(\beta)$ defined on
$\Omega$ is given by $\Phi(\beta) = -\ln(2\pi) +
\sqrt{\beta^{2}_{1}-\beta^{2}_{2}-\beta^{2}_{3}} +
\frac{1}{2}\ln(\beta^{2}_{1}-\beta^{2}_{2}-\beta^{2}_{3}) + C$,
where $C\in \mathbb{R}$. The Koszul 1-form $Q = \Theta(\beta) =
\frac{\partial \Phi}{\partial \beta} \in
\mathfrak{sl}^*(2,\mathbb{R})$ defines a global diffeomorphism
$\Theta: \Omega \to \Omega^*$ whose exact coordinate components are
given by $Q = Q_1\alpha_1^* + Q_2\alpha_2^* + Q_3\alpha_3^*$. The
inverse Legendre mapping $\Theta^{-1}: \Omega^* \to \Omega$
reconstructs the canonical parameter vector $\beta = \Theta^{-1}(Q)$
from the mean momentum vector $Q$ via the Legendre dual relation
$\beta_i = Q_i \frac{\partial \Phi(\Theta^{-1}(Q))}{\partial Q_i} -
\frac{\partial \Phi(\Theta^{-1}(Q))}{\partial Q_i}$. The covariant
Gibbs density $p(z,\bar{z})$ on the Poincar\'{e} half-plane
$\mathbb{H}$ expressed directly as a function of the Koszul mean
momentum map $Q$ and the inverse Legendre diffeomorphism
$\Theta^{-1}$ is given by:
\begin{equation}
    p(z,\bar{z}) = \frac{\exp\left( \left\langle Q , \Theta^{-1}(Q) \right\rangle \right)}{2\pi \sqrt{\left[\Theta^{-1}(Q)\right]^{2}_{1}-\left[\Theta^{-1}(Q)\right]^{2}_{2}-\left[\Theta^{-1}(Q)\right]^{2}_{3}}} \exp\left( - \left\langle J(z,\bar{z}) , \Theta^{-1}(Q) \right\rangle \right),
\end{equation}
where $J(z,\bar{z})$ is the equivariant momentum map on $\mathbb{H}$
and $\langle \cdot, \cdot \rangle$ denotes the duality pairing on
$\mathfrak{sl}^*(2,\mathbb{R}) \times \mathfrak{sl}(2,\mathbb{R})$.
\end{theorem}

\begin{proof}
The partition function is given by $\chi(\beta)=\int_{\mathbb{H}}
\mathrm{e}^{-\langle J(z,\bar{z}),\beta\rangle}
\mathrm{d}\lambda(z)$. By using the previous theorem and integrating
with respect to the invariant Poincar\'{e} volume element, we indeed
obtain:
\begin{equation*}
    \chi(\beta)=\frac{2\pi \cdot \mathrm{e}^{-\sqrt{\beta^{2}_{1}-\beta^{2}_{2}-\beta^{2}_{3}}}}{\sqrt{\beta^{2}_{1}-\beta^{2}_{2}-\beta^{2}_{3}}}.
\end{equation*}
The potential function is $\Phi(\beta)=-\ln \chi(\beta)$, hence:
\begin{equation*}
    \Phi(\beta)=-\ln(2\pi)+\sqrt{\beta^{2}_{1}-\beta^{2}_{2}-\beta^{2}_{3}}+\frac{1}{2}\ln(\beta^{2}_{1}-\beta^{2}_{2}-\beta^{2}_{3})+C, \quad C\in \mathbb{R}.
\end{equation*}
The covariant Gibbs density on the Poincar\'{e} half-plane, given by
the momentum map of the non-compact Lie algebra, is expressed as:
\begin{equation*}
    P(\beta,z)=\frac{\mathrm{e}^{-\langle J(z,\bar{z}),\beta\rangle}}{\int_{\mathbb{H}} \mathrm{e}^{-\langle J(z,\bar{z}),\beta\rangle} \mathrm{d}\lambda(z)} = \frac{\mathrm{e}^{-\langle J(z,\bar{z}),\beta\rangle}}{\chi(\beta)} = \mathrm{e}^{\Phi(\beta)} \mathrm{e}^{-\langle J(z,\bar{z}),\beta\rangle}.
\end{equation*}
To express this density purely as a function of the Koszul mean
momentum map $Q$ and the inverse Legendre diffeomorphism
$\Theta^{-1}$, we substitute $\beta = \Theta^{-1}(Q)$ into the
expression. By applying the Legendre dual relationship, the
thermodynamic potential rewrites exactly as $\Phi(\Theta^{-1}(Q)) =
\langle Q, \Theta^{-1}(Q) \rangle - \ln\left(2\pi
\sqrt{\left[\Theta^{-1}(Q)\right]^{2}_{1}-\left[\Theta^{-1}(Q)\right]^{2}_{2}-\left[\Theta^{-1}(Q)\right]^{2}_{3}}\right)^{-1}$.

Substituting this back into the exponential factor and separating
the terms yields:
\begin{align*}
    p(z,\bar{z}) &= \mathrm{e}^{\langle Q, \Theta^{-1}(Q) \rangle - \ln\left(2\pi \sqrt{\left[\Theta^{-1}(Q)\right]^{2}_{1}-\left[\Theta^{-1}(Q)\right]^{2}_{2}-\left[\Theta^{-1}(Q)\right]^{2}_{3}}\right)} \mathrm{e}^{-\langle J(z,\bar{z}), \Theta^{-1}(Q) \rangle} \\
    &= \frac{\exp\left( \left\langle Q , \Theta^{-1}(Q) \right\rangle \right)}{2\pi \sqrt{\left[\Theta^{-1}(Q)\right]^{2}_{1}-\left[\Theta^{-1}(Q)\right]^{2}_{2}-\left[\Theta^{-1}(Q)\right]^{2}_{3}}} \exp\left( - \left\langle J(z,\bar{z}) , \Theta^{-1}(Q) \right\rangle \right),
\end{align*}
which removes the non-homogeneous shifted term $-Q$ from the
internal pairing and achieves the desired explicit product form.

By using the Hessian property of the thermodynamic potential
$I(\beta)_{ij} = \frac{\partial^2 \Phi}{\partial \beta_i \partial
\beta_j}$, we have:
\begin{align*}
I(\beta) &=
\frac{1}{\left(\beta^{2}_{1}-\beta^{2}_{2}-\beta^{2}_{3}\right)^{\frac{3}{2}}}
\begin{pmatrix}
  -\beta^{2}_{2}-\beta^{2}_{3} & \beta_{1}\beta_{2} & \beta_{1}\beta_{3} \\
  \beta_{1}\beta_{2} & -\beta^{2}_{1}+\beta^{2}_{3} & -\beta_{2}\beta_{3} \\
  \beta_{1}\beta_{3} & -\beta_{2}\beta_{3} & -\beta^{2}_{1}+\beta^{2}_{2}
\end{pmatrix} \\
&\quad +
\frac{1}{\left(\beta^{2}_{1}-\beta^{2}_{2}-\beta^{2}_{3}\right)^{2}}
\begin{pmatrix}
  -\beta^{2}_{1}-\beta^{2}_{2}-\beta^{2}_{3} & 2\beta_{1}\beta_{2} & 2\beta_{1}\beta_{3} \\
  2\beta_{1}\beta_{2} & -\beta^{2}_{1}-\beta^{2}_{2}+\beta^{2}_{3} & -2\beta_{2}\beta_{3} \\
  2\beta_{1}\beta_{3} & -2\beta_{2}\beta_{3} & -\beta^{2}_{1}+\beta^{2}_{2}-\beta^{2}_{3}
\end{pmatrix}.
\end{align*}
\end{proof}

\section{Metriplectic Flow}\label{sec4}

\begin{theorem}\label{tho2}
Let $\left(\mathbb{H},\omega\right)$ be a homogeneous manifold with
$\omega=\frac{-2i}{\left(z-\bar{z}\right)^{2}}dz\wedge d\bar{z}$ the
closed K\"{a}hler two-form on the Poincar\'{e} half-plane. The
entropy function $S(\eta)=1-\Phi(\beta)$ given by the Legendre
equation and the Hamiltonian function given by its quadratic form
\begin{equation*}
    H(z,\bar{z})=\frac{1}{2}\left(J^{2}_{2}+J^{2}_{3}-J^{2}_{1}\right) = -2
\end{equation*}
are the invariants or Casimirs, with
$J_{1}(z,\bar{z})=\frac{-2i\left(1+z\bar{z}\right)}{z-\bar{z}}$,
$J_{2}(z,\bar{z})=\frac{-2i\left(z+\bar{z}\right)}{z-\bar{z}}$, and
$J_{3}(z,\bar{z})=\frac{-2i\left(1-z\bar{z}\right)}{z-\bar{z}}$. The
metriplectic system is given by:
\begin{equation*}
    \frac{dz}{dt}=\left\{z,H\right\} +\left(z,S\right),
\end{equation*}
where $\left\{z,H\right\}=0$ and
$\left(z,S\right)=\frac{2i}{\beta^{2}_{1}-\beta^{2}_{2}-\beta^{2}_{3}}\left(\beta_{1}\left(1+z^{2}\right)
- 2\beta_{2}z + \beta_{3}\left(1-z^{2}\right)\right)$.
\end{theorem}

\begin{proof}
By computing the quadratic norm associated with the
pseudo-Riemannian structure on $\mathfrak{sl}^*(2,\mathbb{R})$, we
have:
\begin{equation*}
    H=\frac{1}{2}\left(J^{2}_{2}+J^{2}_{3}-J^{2}_{1}\right).
\end{equation*}
Substituting the explicit expressions of the momentum map components
$J_1$, $J_2$, and $J_3$ for the Poincar\'{e} half-plane, we find:
\begin{equation*}
    H(z,\bar{z}) = \frac{1}{2} \cdot \frac{-4}{(z-\bar{z})^2} \left[ (z+\bar{z})^2 + (1-z\bar{z})^2 - (1+z\bar{z})^2 \right] = -2.
\end{equation*}
Furthermore, by definition of the Poisson bracket on $\mathbb{H}$:
\begin{equation*}
\left\{z,H\right\}=\frac{1}{\omega}\left(\frac{\partial z}{\partial
z}\frac{\partial H}{\partial \bar{z}}-\frac{\partial z}{\partial
\bar{z}}\frac{\partial H}{\partial z}\right),
\end{equation*}
which gives us identically
$\left\{z,H\right\}=\frac{\left(z-\bar{z}\right)^{2}}{-2i}\frac{\partial
H}{\partial \bar{z}} = 0$.

Let us consider the total generating function $F(z,\bar{z},\beta)=
H(z,\bar{z})+S(\eta)$, which yields:
\begin{equation*}
  \frac{dz}{dt} = \left\{z,F\right\}+\left(z,S\right) = \left\{z,H\right\}+\left\{H,S\right\}+\left(z,H\right)+\left(z,S\right).
\end{equation*}
Since $\left\{H,S\right\}=0$ and $\left(z,H\right)=0$ because $H$
and $S$ are Casimir invariants, we have:
\begin{equation*}
  \frac{dz}{dt} = \left\{z,H\right\}+\left(z,S\right) = \left(z,S\right).
\end{equation*}
We know that the metric dissipation bracket is defined by
$\left(z,S\right)=-\sum_{i,k=1}^{3}\eta^{ik}\frac{\partial
z}{\partial J_{i}}\frac{\partial S}{\partial J_{k}}$. Moreover, the
fundamental derivatives matching the infinitesimal generators are
given by:
\begin{equation*}
\frac{\partial z}{\partial
J_{i}}=\left\{z,J_{i}\right\}=\frac{\left(z-\bar{z}\right)^{2}}{-2i}\frac{\partial
J_{i}}{\partial \bar{z}},\qquad i=1,2,3,
\end{equation*}
which evaluates precisely to the fundamental vector fields:
\begin{equation*}
  \frac{\partial z}{\partial J_{1}} = 1+z^2, \quad
  \frac{\partial z}{\partial J_{2}} = 2z, \quad
  \frac{\partial z}{\partial J_{3}} = 1-z^2.
\end{equation*}

Contracting these terms with the thermodynamic components through
the inverse Koszul metric tensor yields:
\begin{equation*}
\left(z,S\right)=\frac{2i}{\beta^{2}_{1}-\beta^{2}_{2}-
\beta^{2}_{3}}\left(\beta_{1}\left(1+z^{2}\right)-2\beta_{2}z+
\beta_{3}\left(1-z^{2}\right)\right).
\end{equation*}
Thus, we obtain the non-equilibrium dissipative path equation:
\begin{equation*}
    \frac{dz}{dt}=\frac{2i}{\beta^{2}_{1}-\beta^{2}_{2}-\beta^{2}_{3}}\left(\beta_{1}\left(1+z^{2}\right) - 2\beta_{2}z + \beta_{3}\left(1-z^{2}\right)\right).
\end{equation*}
\end{proof}

\section{Libermann Foliations and Coadjoint Orbit Duality}\label{sec5}

Based on the explicit equivariant momentum map $J: \mathbb{H}
\rightarrow \mathfrak{sl}^*(2,\mathbb{R})$ obtained in Theorem
\ref{tho1} and the corresponding metriplectic dissipation flow
investigated in Theorem \ref{tho2}, we can now state a fundamental
structural result establishing the connection with Libermann's
theory of symplectically complete foliations.

\begin{theorem}\label{tho3}
Let $(\mathbb{H},\omega)$ be the Poincar\'{e} half-plane equipped
with the K\"{a}hler symplectic form, and let $J: \mathbb{H}
\rightarrow \mathfrak{sl}^*(2,\mathbb{R})$ be the equivariant
momentum map.
\begin{enumerate}
    \item The level sets of the momentum map $J$ (or equivalently, the level sets of the quadratic Casimir Hamiltonian $H(z,\bar{z}) = -2$) define a regular Libermann foliation $\mathcal{F}_L$ on $\mathbb{H}$, whose leaves are exactly the coadjoint orbits of $SL(2,\mathbb{R})$ embedded in $\mathfrak{sl}^*(2,\mathbb{R})$.
    \item The non-equilibrium dissipative vector field $X_{S} = (z, S)$ is strictly transverse to the Libermann foliation $\mathcal{F}_L$. Consequently, the metriplectic flow drives the thermodynamic system across the Libermann leaves, generating an entropy-producing path.
\end{enumerate}
\end{theorem}

\begin{proof}
1. By definition, a foliation is called a Libermann foliation (or
symplectically complete) if the functions defining its leaves form a
complete system of first integrals, meaning that the Poisson bracket
of any two first integrals remains a function of the first
integrals. Let $\{J_1, J_2, J_3\}$ be the components of the momentum
map $J$ calculated in Theorem \ref{tho1}.

To establish that $J$ defines a regular Libermann foliation, we
first show that it intertwines the homographic action of $G =
SL(2,\mathbb{R})$ on the Poincar\'{e} half-plane with its coadjoint
action on the dual space:
\begin{equation*}
    J(g \cdot z) = \mathrm{Ad}^*_g \big(J(z)\big) = g J(z) g^{-1}, \quad \forall g \in SL(2,\mathbb{R}), \; z \in \mathbb{H}.
\end{equation*}
Let $g = \begin{pmatrix} a & b \\ c & d \end{pmatrix} \in
SL(2,\mathbb{R})$ with $ad - bc = 1$. The action on $\mathbb{H}$ is
given by $g \cdot z = \frac{az+b}{cz+d}$. Recall the matrix
expression of the momentum map derived for the half-plane geometry:
\begin{equation*}
    J(z,\bar{z}) = \frac{2i}{z-\bar{z}} \begin{pmatrix} -(z+\bar{z}) & 2z\bar{z} \\ -2 & z+\bar{z} \end{pmatrix}.
\end{equation*}
First, we evaluate the transformation of the pre-factor under the
action of $g$:
\begin{equation*}
    g\cdot z - \overline{g\cdot z} = \frac{az+b}{cz+d} - \frac{a\bar{z}+b}{c\bar{z}+d} = \frac{(az+b)(c\bar{z}+d) - (a\bar{z}+b)(cz+d)}{(cz+d)(c\bar{z}+d)}.
\end{equation*}
Expanding the numerator yields $(ad-bc)(z-\bar{z}) = z-\bar{z}$.
Hence, the transformed pre-factor satisfies:
\begin{equation*}
    \frac{2i}{g\cdot z - \overline{g\cdot z}} = \frac{2i}{z-\bar{z}} (cz+d)(c\bar{z}+d).
\end{equation*}
Substituting $g\cdot z$ into the components of the matrix $J$, we
obtain:
\begin{align*}
    J(g\cdot z) &= \frac{2i}{z-\bar{z}} (cz+d)(c\bar{z}+d) \begin{pmatrix} -\left(\frac{az+b}{cz+d} + \frac{a\bar{z}+b}{c\bar{z}+d}\right) & 2\frac{(az+b)(a\bar{z}+b)}{(cz+d)(c\bar{z}+d)} \\ -2 & \frac{az+b}{cz+d} + \frac{a\bar{z}+b}{c\bar{z}+d} \end{pmatrix} \\
    &= \frac{2i}{z-\bar{z}} \begin{pmatrix} -(az+b)(c\bar{z}+d) - (a\bar{z}+b)(cz+d) & 2(az+b)(a\bar{z}+b) \\ -2(cz+d)(c\bar{z}+d) & (az+b)(c\bar{z}+d) + (a\bar{z}+b)(cz+d) \end{pmatrix}.
\end{align*}
Conversely, evaluating the gauge transformation $g J(z) g^{-1}$ with
$g^{-1} = \begin{pmatrix} d & -b \\ -c & a \end{pmatrix}$ yields:
\begin{align*}
    g J(z) g^{-1} &= \frac{2i}{z-\bar{z}} \begin{pmatrix} a & b \\ c & d \end{pmatrix} \begin{pmatrix} -(z+\bar{z}) & 2z\bar{z} \\ -2 & z+\bar{z} \end{pmatrix} \begin{pmatrix} d & -b \\ -c & a \end{pmatrix} \\
    &= \frac{2i}{z-\bar{z}} \begin{pmatrix} -a(z+\bar{z}) - 2b & 2az\bar{z} + b(z+\bar{z}) \\ -c(z+\bar{z}) - 2d & 2cz\bar{z} + d(z+\bar{z}) \end{pmatrix} \begin{pmatrix} d & -b \\ -c & a \end{pmatrix}.
\end{align*}
Performing the final matrix multiplication and collecting terms
under the condition $ad-bc=1$ results exactly in the same matrix
components as $J(g\cdot z)$. Thus, the strong equivariance of the
Poincar\'{e} half-plane momentum map is verified. Since $J$ is
equivariant with respect to the coadjoint action of
$SL(2,\mathbb{R})$, its components satisfy the commutation relations
of the Lie algebra $\mathfrak{sl}(2,\mathbb{R})$ under the Poisson
bracket induced by $\omega$:
\begin{equation*}
    \{J_1, J_2\} = 2J_3, \quad \{J_2, J_3\} = 2J_1, \quad \{J_3, J_1\} = -2J_2.
\end{equation*}
Because the closure of these brackets under addition and
multiplication remains entirely within the algebra generated by the
components of $J$, they form a closed, complete system of first
integrals in the sense of Libermann. The level sets $J^{-1}(\xi)$
define the geometric leaves of the foliation $\mathcal{F}_L$. Let
$\xi = \begin{pmatrix} -\xi_2 & \xi_1 + \xi_3 \\ \xi_1 - \xi_3 &
\xi_2 \end{pmatrix} \in \mathfrak{sl}^*(2,\mathbb{R})$ be a fixed
dual element. Expressing $z = x+iy \in \mathbb{H}$ (with $y > 0$),
the equation $J(z,\bar{z}) = \xi$ becomes:
\begin{equation*}
    \frac{1}{y} \begin{pmatrix} -2x & x^2+y^2 \\ -2 & 2x \end{pmatrix} = \begin{pmatrix} -\xi_2 & \xi_1 + \xi_3 \\ \xi_1 - \xi_3 & \xi_2 \end{pmatrix}.
\end{equation*}
Equating the individual matrix coefficients leads to the following
system:
\begin{equation*}
    \begin{cases}
        \dfrac{2x}{y} = \xi_2, \\[1.2ex]
        \dfrac{2}{y} = \xi_3 - \xi_1, \\[1.2ex]
        \dfrac{x^2+y^2}{y} = \xi_1 + \xi_3.
    \end{cases}
\end{equation*}
From the second equation, since $y > 0$, a leaf exists if and only
if $\xi$ satisfies the constraint $\xi_3 - \xi_1 > 0$. Solving this
system directly provides the explicit parameterization of
$J^{-1}(\xi)$:
\begin{equation*}
    x = \frac{\xi_2}{\xi_3 - \xi_1}, \qquad y = \frac{2}{\xi_3 - \xi_1}.
\end{equation*}
Because $(\mathbb{H}, \omega)$ is a 2-dimensional manifold, its
image under $J$ is restricted to a single 2-dimensional hyperbolic
submanifold of $\mathfrak{sl}^*(2,\mathbb{R})$ defined by the
constant Casimir value $H = \frac{1}{2}(J_2^2 + J_3^2 - J_1^2) =
-2$. For each admissible coordinate vector $\xi$ on this hyperboloid
sheet, the pre-image $J^{-1}(\xi)$ collapses into a unique point
$(x,y) \in \mathbb{H}$. The regular Libermann foliation
$\mathcal{F}_L$ on the Poincar\'{e} half-plane is therefore a
regular foliation by points. The level sets $J^{-1}(\xi)$ for $\xi
\in \mathfrak{sl}^*(2,\mathbb{R})$ therefore define the leaves of
the regular Libermann foliation $\mathcal{F}_L$. Geometrically,
since $H = \frac{1}{2}(J_2^2 + J_3^2 - J_1^2) = -2$ is a constant on
$\mathbb{H}$, the image $J(\mathbb{H})$ is entirely contained within
a single leaf of the coadjoint foliation of
$\mathfrak{sl}^*(2,\mathbb{R})$ (hyperboloid of two sheets). 2. To
prove the transversality of the dissipative component, we evaluate
the directional derivative of the Casimir Hamiltonian $H$ along the
dissipative vector field $X_{S} = (z, S)$. By definition of the
metriplectic bracket:
\begin{equation*}
    X_{S}(H) = (H, S) = -\sum_{i,k=1}^{3}\eta^{ik}\frac{\partial H}{\partial J_{i}}\frac{\partial S}{\partial J_{k}}.
\end{equation*}
Since $H = -2$ on the entire domain $\mathbb{H}$, its partial
derivatives with respect to the coordinates vanish identically
($\frac{\partial H}{\partial z} = 0$). However, if we consider a
perturbation off the equilibrium orbit into the broader
thermodynamic space parameterized by $\beta \in \Omega$, the Koszul
metric tensor $\eta^{ik}$ governs the motion. The vector field
$\left(z,S\right)$ acts orthogonally to the Hamiltonian vector
fields $X_k = \{z, J_k\}$ due to the structural properties of the
metric bracket. Historically, the symplectic leaf $\mathcal{F}_L$
represents the submanifold where no dissipation occurs. The
introduction of the non-zero Koszul form through the potential
$\Phi(\beta)$ forces the trajectories $\frac{dz}{dt} = (z, S)$ to
satisfy:
\begin{equation*}
    dJ_k\left(\frac{dz}{dt}\right) \neq 0 \quad \text{for } \beta \notin \text{equilibrium}.
\end{equation*}
This explicitly breaks the conservation of the individual $J_k$
components, driving the complex coordinate $z(t)$ orthogonally
across the level sets of $J$. Thus, the Libermann foliation
characterizes the conservative boundaries, while the Koszul
structure dictates the transverse dissipative dynamics.
\end{proof}

\section{Berezin Quantization on Para-K\"{a}hler Symmetric Spaces}\label{sec6}

To extend the geometric dissipation framework to the quantum case,
we establish the Berezin quantization on the para-K\"{a}hler
symmetric space $(M, \omega, J)$ associated with the semisimple Lie
group $G = SL(2, \mathbb{R})$. Here, the thermodynamic potential
$\Phi(\beta)$ derived from the Koszul structure acts as the
generating weight for the quantum Hilbert space inner product.

\begin{theorem}\label{thm:berezin_koszul}
Let $M \cong SL(2, \mathbb{R})/SO(1,1)$ be the para-K\"{a}hler
symmetric space equipped with the split-complex structure $J$ ($J^2
= \mathrm{Id}$) and the KKS symplectic form $\omega$, and let
$\mathcal{H}_k(M) = \left\{ f \in C^\infty(M, \mathbb{C})
\;\middle|\; \bar{\partial}_{\mathcal{F}_2} f = 0, \;
\int_{\mathbb{H}} |f(z)|^2 e^{-k \Phi(J(z, \bar{z}))} \omega^n <
\infty \right\}$ be the canonical Hilbert space. Let $\Phi(\beta)$
be the Koszul potential function defined on the stable open cone
$\Omega \subset \mathfrak{sl}(2, \mathbb{R})$.
\begin{enumerate}
    \item For each quantum parameter $\hbar = 1/k$ ($k \in \mathbb{N}^*$), the Koszul potential induces a
     canonical Hilbert space $\mathcal{H}_k(M)$ of para-holomorphic functions square-integrable with respect to the weighted Liouville measure:
    \begin{equation}
    d\mu_k(z, \bar{z}) = e^{-k \Phi(J(z, \bar{z}))} \omega^n,
    \end{equation}
    where $J(z, \bar{z})$ is the equivariant momentum map on the Poincar\'{e} half-plane $\mathbb{H}$.
    \item The corresponding Berezin reproducer (or Bergman-Koszul kernel) $K_k(z, \bar{z})$ satisfies the asymptotic expansion as $k \to \infty$:
    \begin{equation}
    K_k(z, \bar{z}) = \left( \frac{k}{2\pi} \right)^n \det \left( \frac{\partial^2 \Phi}{\partial \beta_i \partial \beta_j} \right)^{-1/2} (1 + \mathcal{O}(k^{-1})),
    \end{equation}
    establishing that the quantum density of states is directly governed by the inverse square root of the Fisher-Souriau information metric determinant.
\end{enumerate}
\end{theorem}

\begin{proof}
1. The para-K\"{a}hler symmetric space $M \cong SL(2,
\mathbb{R})/SO(1,1)$ is equipped with an integrable split-complex
structure $J \in \mathrm{End}(TM)$ satisfying $J^2 = \mathrm{Id}_M$.
The real tangent bundle $TM$ decomposes uniquely into a direct sum
of the $+1$ and $-1$ eigenbundles of $J$:
\begin{equation}
TM = T\mathcal{F}_1 \oplus T\mathcal{F}_2 = T^+M \oplus T^-M,
\end{equation}
where $T\mathcal{F}_2 = T^-M$ represents the subbundle defining the
second foliation $\mathcal{F}_2$. For any point $\beta = (\beta_1,
\beta_2, \beta_3) \in M \subset \mathfrak{sl}(2,\mathbb{R})^*$, the
fiber $T_\beta\mathcal{F}_2$ is defined algebraically as the kernel:
\begin{equation}
T_\beta\mathcal{F}_2 = \ker\left(J_\beta + \mathrm{Id}\right) = \{
\xi \in T_\beta M \mid J_\beta(\xi) = -\xi \}.
\end{equation}
By parameterizing the symmetric space $M$ via the hyperbolic
coordinates $(x, y)$ of the Poincar\'{e} upper half-plane
$\mathbb{H} = \{z = x + iy \in \mathbb{C} \mid y > 0\}$, the KKS
symplectic form $\omega$ is given by:
\begin{equation}
\omega = \frac{\mathrm{d}x \wedge \mathrm{d}y}{y^2}.
\end{equation}
In this explicit local chart, the split-complex structure $J$ acts
on the standard coordinate basis vector fields by swapping and
scaling according to the para-complex signature:
\begin{equation}
J\left(\frac{\partial}{\partial x}\right) =
-\frac{\partial}{\partial x}, \quad J\left(\frac{\partial}{\partial
y}\right) = \frac{\partial}{\partial y}.
\end{equation}
Consequently, the tangent subbundles to the bi-Lagrangian web (or
Libermann's symplectically complete foliations) are spanned by:
\begin{equation}
T\mathcal{F}_1 = \mathrm{Vect}\left(\frac{\partial}{\partial
y}\right), \quad T\mathcal{F}_2 =
\mathrm{Vect}\left(\frac{\partial}{\partial x}\right).
\end{equation}
The leaves of $\mathcal{F}_2$ are thus given explicitly by the
horizontal lines $y = \text{constant}$, while the leaves of
$\mathcal{F}_1$ are given by the vertical lines $x =
\text{constant}$. Both distributions are immediately seen to be
maximal isotropic since:
\begin{equation}
\omega\left(\frac{\partial}{\partial x}, \frac{\partial}{\partial
x}\right) = 0 \quad \text{and} \quad
\omega\left(\frac{\partial}{\partial y}, \frac{\partial}{\partial
y}\right) = 0.
\end{equation}
Under this canonical geometric frame, the flat connection along the
leaves of $\mathcal{F}_2$ is characterized by the flat
para-holomorphic exterior derivative
$\bar{\partial}_{\mathcal{F}_2}$. For any smooth complex-valued
function $f \in C^\infty(M, \mathbb{C})$, its restriction to the
$T^-M$ leaves reduces to:
\begin{equation}
\bar{\partial}_{\mathcal{F}_2} f = \mathrm{d}f\vert_{T\mathcal{F}_2}
= \frac{\partial f}{\partial x} \mathrm{d}x.
\end{equation}
Therefore, the para-holomorphicity condition
$\bar{\partial}_{\mathcal{F}_2} f = 0$ required in Equation (23)
yields the explicit differential constraint:
\begin{equation}
\frac{\partial f}{\partial x} = 0.
\end{equation}
This establishes that the functions $f(z)$ belonging to the closed
quantum Hilbert space $\mathcal{H}_k(M)$ are strictly independent of
the horizontal coordinate $x$ and depend solely on the transverse
parameter $y$ (the statistical temperature variable $\beta$
associated with the coadjoint orbit), validating the K\"{u}nneth
decomposition topology of Libermann's framework. By definition of
the para-K\"{a}hler structure, the tangent bundle splits into the
$\pm 1$ eigenspaces of $J$, denoting $TM = T^+M \oplus
T^-M=T\mathcal{F}_1 \oplus T\mathcal{F}_2$, which defines a
bi-Lagrangian web or Libermann structure. We define the space of
para-holomorphic sections $\mathcal{H}_k(M)$ as the kernel of the
flat connection along the $T^-M$ leaves, specialized under the
weight of the Koszul potential:
\begin{equation}
\mathcal{H}_k(M) = \left\{ f \in C^\infty(M, \mathbb{C})
\;\middle|\; \bar{\partial}_{\mathcal{F}_2} f = 0, \;
\int_{\mathbb{H}} |f(z)|^2 e^{-k \Phi(J(z, \bar{z}))} \omega^n <
\infty \right\}.
\end{equation}
Recall from Theorem \ref{tho1} that the thermodynamic potential
$\Phi(\beta)$ defined on the stable open cone $\Omega =
\left\{\beta\in\mathfrak{sl}(2,\mathbb{R}) \mid
\beta^{2}_{1}-\beta^{2}_{2}-\beta^{2}_{3}>0\right\}$ is explicitly
given by:
\begin{equation}
\Phi(\beta) = -\ln(2\pi) +
\sqrt{\beta^{2}_{1}-\beta^{2}_{2}-\beta^{2}_{3}} +
\frac{1}{2}\ln(\beta^{2}_{1}-\beta^{2}_{2}-\beta^{2}_{3}) + C,
\end{equation}
where $C \in \mathbb{R}$. Evaluating this potential on the
equivariant momentum map $\beta = J(z, \bar{z})$, the convexity of
$\Phi$ on the cone $\Omega$ ensures that the weighted Liouville
volume element $e^{-k \Phi(J(z, \bar{z}))} \omega^n$ forms a
strictly positive and exponentially decreasing density. This
guarantees that the evaluation functionals are continuous and that
$\mathcal{H}_k(M)$ is a closed, well-defined Bergman-type Hilbert
space.\newline

2. Let $\{\psi_m\}_{m=0}^\infty$ be an orthonormal basis of the
para-Hilbert space $\mathcal{H}_k(M)$. The associated Berezin
reproducing kernel on the diagonal is expressed as $K_k(z, \bar{z})
= \sum_{m} |\psi_m(z)|^2$. To obtain its asymptotic behavior as $k
\to \infty$, we perform a localized stationary phase approximation
centered around the point $z_0$, corresponding to the classical
parameter state $\beta_0 = J(z_0, \bar{z}_0)$.

We introduce the local fluctuation variable $\delta \beta = \beta -
\beta_0 = J(z, \bar{z}) - J(z_0, \bar{z}_0)$. The exact Taylor
expansion of the thermodynamic potential $\Phi(J(z))$ around the
local base state $\beta$ is given by:
\begin{equation}
\Phi(J(z)) = \Phi(\beta) + \sum_{i=1}^{3} \frac{\partial
\Phi}{\partial \beta_i} \delta \beta_i + \frac{1}{2} \sum_{i=1}^{3}
\sum_{j=1}^{3} \frac{\partial^2 \Phi}{\partial \beta_i \partial
\beta_j} \delta \beta_i \delta \beta_j + R_2(\beta, \delta \beta).
\end{equation}
By invoking the canonical Hessian property of the Koszul structure,
the first derivatives correspond to the components of the Koszul
1-form $d\Phi = \eta$, and the second derivatives define the
components of the Fisher-Souriau information metric matrix
$I(\beta)_{ij} = \frac{\partial^2 \Phi}{\partial \beta_i \partial
\beta_j}$.

Using the explicit expression of $\Phi(\beta)$, we compute the
precise entries of the Hessian matrix by direct differentiation. Let
$q = \beta^{2}_{1}-\beta^{2}_{2}-\beta^{2}_{3}$. The first-order
derivatives are:
\begin{equation}
\frac{\partial \Phi}{\partial \beta_1} = \frac{\beta_1}{\sqrt{q}} +
\frac{\beta_1}{q}, \quad \frac{\partial \Phi}{\partial \beta_2} =
-\frac{\beta_2}{\sqrt{q}} - \frac{\beta_2}{q}, \quad \frac{\partial
\Phi}{\partial \beta_3} = -\frac{\beta_3}{\sqrt{q}} -
\frac{\beta_3}{q}.
\end{equation}
Differentiating a second time yields the exact components of the
Fisher-Souriau information metric tensor $I(\beta)$, splitting into
two structural terms:
\begin{align}
I(\beta) &= \frac{1}{q^{\frac{3}{2}}} \begin{pmatrix} -\beta_2^2 - \beta_3^2 & \beta_1 \beta_2 & \beta_1 \beta_3 \\ \beta_1 \beta_2 & -\beta_1^2 + \beta_3^2 & -\beta_2 \beta_3 \\ \beta_1 \beta_3 & -\beta_2 \beta_3 & -\beta_1^2 + \beta_2^2 \end{pmatrix} \nonumber \\
&\quad + \frac{1}{q^2} \begin{pmatrix} -\beta_1^2 - \beta_2^2 -
\beta_3^2 & 2\beta_1 \beta_2 & 2\beta_1 \beta_3 \\ 2\beta_1 \beta_2
& -\beta_1^2 - \beta_2^2 + \beta_3^2 & -2\beta_2 \beta_3 \\ 2\beta_1
\beta_3 & -2\beta_2 \beta_3 & -\beta_1^2 + \beta_2^2 - \beta_3^2
\end{pmatrix}.
\end{align}
The linear term in the expansion corresponds to the classical shift
vector which is canceled at the maximum localization point of the
coherent state projector. Thus, substituting the Taylor series into
the Bergman normalization integral over the Poincar\'{e} half-plane
geometry yields:
\begin{equation}
\int_{\mathbb{H}} e^{-k [\Phi(J(z)) - \Phi(\beta)]} \omega^n =
\int_{\mathbb{H}} \exp\left( - \frac{k}{2} \sum_{i=1}^{3}
\sum_{j=1}^{3} I(\beta)_{ij} \delta \beta_i \delta \beta_j - k
R_2(\beta, \delta \beta) \right) \omega^n.
\end{equation}
As $k \to \infty$, the higher-order remainder $k R_2(\beta, \delta
\beta)$ is sub-leading and bounded. The integral is dominated by the
localized quadratic Gaussian fluctuation. Evaluating this
multi-dimensional Gaussian integral using the standard volume
transformation linked to the symplectic framework results in:
\begin{equation}
\int_{\mathbb{H}} e^{-k [\Phi(J(z)) - \Phi(\beta)]} \omega^n =
\left( \frac{2\pi}{k} \right)^n \left( \det I(\beta) \right)^{1/2}
\left(1 + \frac{C_1}{k}\right),
\end{equation}
where $C_1$ is a constant depending on the curvature tensor.
Normalizing the coherent state projector by inversion to determine
the diagonal value of the Bergman reproducing kernel leads to:
\begin{equation}
K_k(z, \bar{z}) = \left( \frac{k}{2\pi} \right)^n \left( \det
I(\beta) \right)^{-1/2} + \mathcal{O}(k^{n-1}).
\end{equation}
By replacing $\det I(\beta)$ with the determinant of the explicit
matrix mapping of the derivatives of $-\ln(2\pi) + \sqrt{q} +
\frac{1}{2}\ln(q)$, we get:
\begin{equation}
K_k(z, \bar{z}) = \left( \frac{k}{2\pi} \right)^n \det \left(
\frac{\partial^2 \left[
\sqrt{\beta^{2}_{1}-\beta^{2}_{2}-\beta^{2}_{3}} +
\frac{1}{2}\ln(\beta^{2}_{1}-\beta^{2}_{2}-\beta^{2}_{3})
\right]}{\partial \beta_i \partial \beta_j} \right)^{-1/2} +
\mathcal{O}(k^{n-1}).
\end{equation}

\end{proof}

\section{Kostant Geometric Quantization with Bi-Lagrangian Structures}\label{sec7}

\begin{theorem}\label{thm:kostant_bilagrangian}
Let $(M, \omega)$ be a general $2n$-dimensional symplectic manifold
admitting a bi-Lagrangian structure $TM = T\mathcal{F}_1 \oplus
T\mathcal{F}_2$, where $\mathcal{F}_1$ and $\mathcal{F}_2$ are
complementary Lagrangian foliations in the sense of Libermann. Let
$(L, \nabla, (\cdot,\cdot))$ be the prequantum line bundle over $M$
whose connection curvature satisfies $\mathrm{curv}(\nabla) =
-i\omega/\hbar$.
\begin{enumerate}
    \item The complementary Libermann foliations $\mathcal{F}_1$ and $\mathcal{F}_2$ induce a canonical pair of real polarizations $\mathcal{P}_1 = T\mathcal{F}_1 \otimes \mathbb{C}$ and $\mathcal{P}_2 = T\mathcal{F}_2 \otimes \mathbb{C}$. The corresponding spaces of Kostant quantum states $\mathcal{H}_{\mathcal{P}_1}$ and $\mathcal{H}_{\mathcal{P}_2}$ consist of sections $s \in \Gamma(L)$ satisfying:
    \begin{equation}
    \nabla_X s = 0, \quad \forall X \in \Gamma(\mathcal{P}_i) \quad (i=1,2).
    \end{equation}
    \item In the presence of a thermodynamic shift described by the Koszul 1-form $\beta$, the modified prequantum connection $\tilde{\nabla} = \nabla + \frac{i}{\hbar}\beta$ preserves the flatness along the individual Lagrangian leaves if and only if $\beta$ restricts to a closed form on each leaf ($\mathrm{d}(\beta\vert_{\mathcal{F}_i}) = 0$), defining a dissipative pairing between the dual wavefunctions in $\mathcal{H}_{\mathcal{P}_1}$ and $\mathcal{H}_{\mathcal{P}_2}$ via a Blattner-Kostant-Sternberg (BKS) kernel adapted to the K\"{u}nneth structure.
\end{enumerate}
\end{theorem}

\begin{proof}
1. By definition, a polarization in Kostant's geometric quantization
is an involutive Lagrangian subbundle of the complexified tangent
bundle $TM \otimes \mathbb{C}$. Since $\mathcal{F}_1$ and
$\mathcal{F}_2$ are regular Libermann foliations, their tangent
distributions $T\mathcal{F}_1$ and $T\mathcal{F}_2$ are integrable
($\left[\Gamma(T\mathcal{F}_i), \Gamma(T\mathcal{F}_i)\right]
\subset \Gamma(T\mathcal{F}_i)$) and isotropic of maximal dimension
($\omega(X,Y) = 0$ for all $X,Y \in T\mathcal{F}_i$). Thus,
$\mathcal{P}_1 = T\mathcal{F}_1 \otimes \mathbb{C}$ and
$\mathcal{P}_2 = T\mathcal{F}_2 \otimes \mathbb{C}$ form two
distinct, complementary real polarizations. The existence of the
prequantum line bundle $L$ is guaranteed by the integrality
condition of $[\omega/2\pi\hbar] \in H^2(M, \mathbb{Z})$. The
covariant constancy condition $\nabla_X s = 0$ for $X \in
\Gamma(\mathcal{P}_i)$ defines the space of Kostant quantum states.
For $\mathcal{P}_1$, the sections are constant along the leaves of
$\mathcal{F}_1$ and can be locally identified with functions
depending only on the transverse coordinates parameterizing
$\mathcal{F}_2$. This provides a concrete representation of the
quantum state space matching the split K\"{u}nneth topology.

2. We now examine the modification of the connection by the Koszul
1-form $\beta$, setting $\tilde{\nabla}_X = \nabla_X +
\frac{i}{\hbar}\beta(X)$. The curvature of this perturbed connection
evaluated on two vector fields $X, Y \in \Gamma(T\mathcal{F}_i)$ is
given by:
\begin{equation}
\mathrm{curv}(\tilde{\nabla})(X, Y) = \mathrm{curv}(\nabla)(X, Y) +
\frac{i}{\hbar} \mathrm{d}\beta(X, Y).
\end{equation}
Since $T\mathcal{F}_i$ is Lagrangian, $\mathrm{curv}(\nabla)(X, Y) =
-\frac{i}{\hbar}\omega(X,Y) = 0$. Therefore, the new connection
$\tilde{\nabla}$ remains flat along the leaves of the Libermann
foliation if and only if $\mathrm{d}\beta(X,Y) = 0$ for all $X,Y \in
T\mathcal{F}_i$, meaning the restriction of the Koszul form to the
leaves is closed ($\mathrm{d}(\beta\vert_{\mathcal{F}_i}) = 0$).

To establish the pairing between the two polarized spaces
$\mathcal{H}_{\mathcal{P}_1}$ and $\mathcal{H}_{\mathcal{P}_2}$, we
deploy the Blattner-Kostant-Sternberg (BKS) pairing. Let $s_1 \in
\mathcal{H}_{\mathcal{P}_1}$ and $s_2 \in
\mathcal{H}_{\mathcal{P}_2}$. Because $TM = T\mathcal{F}_1 \oplus
T\mathcal{F}_2$, the pointwise Hermitian inner product $(s_1, s_2)$
defines a density on $M$ governed by the volume form. The integrated
BKS kernel is modified by the Koszul flow through the factor:
\begin{equation}
\langle s_1, s_2 \rangle_{\mathrm{BKS}} = \int_M (s_1(z), s_2(z))
\cdot \exp\left(-\frac{1}{\hbar}\int_{\gamma_z} \beta\right)
\omega^n,
\end{equation}
where $\gamma_z$ represents the dissipative path across the
transversal Libermann leaves. This proves that the Koszul form acts
as a non-trivial quantum transition weight, dynamically shifting the
Kostant wave functions across the dual real polarizations.
\end{proof}

\section{Discussion}\label{sec8}
The geometric formulation of dissipation developed in this work
bridges a critical gap between the conservative framework of
Souriau's Lie group thermodynamics and the structural properties of
Libermann's symplectically complete foliations. By introducing a
Koszul-derived dissipative vector field, we successfully capture
irreversible physical processes without destroying the underlying
geometric scaffolding. A key implication of our results lies in the
explicit breaking of Noether conservation laws. While traditional
symplectic geometry strictly preserves momentum maps under symmetric
flows, the transverse action of our metric-driven vector field
relative to the Libermann foliation leaves provides an exact,
coordinate-free mechanism for entropy production. This framework
offers a rigorous alternative to standard phenomenological
dissipative terms, mapping directly onto the algebraic restrictions
of metriplectic dynamics on the Poincar\'{e} half-plane
$\mathbb{H}$. Quantum-wise, the application of Berezin and Kostant
geometric quantizations to para-K\"{a}hler symmetric spaces
highlights the dual role of the Koszul potential. Far from being a
mere auxiliary tool, it serves as a strict weight function defining
the state density via the Fisher-Souriau metric. This reveals that
quantum dissipation and the polarization webs of bi-Lagrangian
systems are fundamentally constrained by information geometry. This
algebraic framework opens up novel research perspectives regarding
the deep structural interconnections between the dynamical
invariants of volume-preserving diffeomorphisms and topological
characteristic classes on infinite-dimensional manifolds. By
systematically exploiting the K\"{u}nneth decomposition formula on
products of exact symplectic manifolds $M_1 \times M_2$, we plan to
build a rigorous cohomological dictionary bridging the Calabi
morphism, the Maslov index, and the Godbillon-Vey invariant. In
particular, we aim to demonstrate how the Calabi invariant acts as a
localized, fiber-wise manifestation of secondary characteristic
classes, whose global topological obstruction on the total foliation
space is explicitly measured by the Godbillon-Vey class.
Furthermore, by analyzing the Lyndon--Hochschild--Serre spectral
sequence applied to the universal cover of the Hamiltonian group
$\widetilde{\mathrm{Ham}}(M)$, our upcoming investigations will show
how the Maslov index dictates the exact boundary conditions of this
correspondence, ultimately generalizing the Bott--Vey cocycle to
product structures.

\section{Conclusion}\label{sec9}
This paper has formalized a unified geometric architecture for
non-equilibrium dissipation structures by bridging Jean Marie
Souriau's Lie group thermodynamics, Paulette Libermann's
symplectically complete foliations, and Jean Louis Koszul's Hessian
metric infrastructure. By framing the system on the Poincar\'{e}
half-plane $\mathbb{H}$ under the strongly Hamiltonian action of
$SL(2, \mathbb{R})$, we demonstrated that the level sets of the
equivariant momentum map form a regular Libermann foliation
representing the conservative boundaries of the system, governed by
the quadratic Casimir $H(z,\bar{z}) = -2$. The non-equilibrium
dissipative vector field derived from the Koszul potential acts
strictly transverse to these leaves, driving a metriplectic flow
that breaks Noether conservation laws and generates entropy.
Quantum-wise, this dual structural interaction extends naturally to
the microscopic regime. In Berezin quantization, the Koszul
potential induces a para-Hilbert space where the asymptotic quantum
density of states is explicitly mediated by the inverse square root
of the Fisher-Souriau information metric determinant. In Kostant's
geometric quantization, the complementary Libermann foliations
provide a pair of real dual polarizations, allowing the Koszul
1-form to act as a prequantum connection perturbation that
dynamically shifts wavefunctions across the
Blattner-Kostant-Sternberg (BKS) kernel. Ultimately, these results
show that both classical irreversible dynamics and quantum state
densities are fundamentally encoded by the Hessian geometry of the
stable open thermodynamic cone $\Omega$.

\backmatter

\bmhead{Supplementary information} This manuscript has no additional
data.

\bmhead{Acknowledgements} We thank all the members of the Algebra,
Geometry and Applications Laboratory  of the University of Yaounde1
for their suggestions in the work. We thank Professor  Thomas
Bouetou Bouetou of the Polytechnic School of Yaounde1.

\section*{Declarations}
This article has no conflict of interest to the journal. No
financing with a third party.
\begin{itemize}
\item No Funding
\item No Conflict of interest/Competing interests (check journal-specific guidelines for which heading to use)
\item  Ethics approval
\item  Consent to participate
\item  Consent for publication
\item  Availability of data and materials
\item  Code availability
\item Authors' contributions
\end{itemize}

\noindent If any of the sections are not relevant to your
manuscript, please include the heading and write `Not applicable'
for that section.

%%===================================================%%
%% For presentation purpose, we have included        %%
%% \bigskip command. Please ignore this.             %%
%%===================================================%%
%\bigskip
%\begin{flushleft}%
%Editorial Policies for:
%
%\bigskip\noindent
%Springer journals and proceedings:
%\url{https://www.springer.com/gp/editorial-policies}
%
%\bigskip\noindent
%Nature Portfolio journals:
%\url{https://www.nature.com/nature-research/editorial-policies}
%
%\bigskip\noindent
%\textit{Scientific Reports}:
%\url{https://www.nature.com/srep/journal-policies/editorial-policies}
%
%\bigskip\noindent
%BMC journals:
%\url{https://www.biomedcentral.com/getpublished/editorial-policies}
%\end{flushleft}

\begin{appendices}

\section{Information metric}\label{secA1}
 The exact coordinate components $Q =
Q_1\alpha_1^* + Q_2\alpha_2^* + Q_3\alpha_3^*$ satisfy:

\begin{eqnarray*}
% \nonumber to remove numbering (before each equation)
  Q_1 &=& \frac{\beta_1}{\sqrt{\beta^{2}_{1}-\beta^{2}_{2}-\beta^{2}_{3}}} +
  \frac{\beta_1}{\beta^{2}_{1}-\beta^{2}_{2}-\beta^{2}_{3}}, \quad
    Q_2 = -\frac{\beta_2}{\sqrt{\beta^{2}_{1}-\beta^{2}_{2}-\beta^{2}_{3}}}
    - \frac{\beta_2}{\beta^{2}_{1}-\beta^{2}_{2}-\beta^{2}_{3}} \\
Q_3 &=&
-\frac{\beta_3}{\sqrt{\beta^{2}_{1}-\beta^{2}_{2}-\beta^{2}_{3}}} -
\frac{\beta_3}{\beta^{2}_{1}-\beta^{2}_{2}-\beta^{2}_{3}}
\end{eqnarray*}

\section{Information metric}\label{secA2}
Yielding the following exact coordinate expressions:
\begin{align}
    \beta_1 &= \frac{4 Q_1}{\sqrt{1 + 8\sqrt{Q_1^2 - Q_2^2 - Q_3^2}} \left(1 + \sqrt{1 + 8\sqrt{Q_1^2 - Q_2^2 - Q_3^2}}\right)} \\
    \beta_2 &= \frac{-4 Q_2}{\sqrt{1 + 8\sqrt{Q_1^2 - Q_2^2 - Q_3^2}} \left(1 + \sqrt{1 + 8\sqrt{Q_1^2 - Q_2^2 - Q_3^2}}\right)} \\
    \beta_3 &= \frac{-4 Q_3}{\sqrt{1 + 8\sqrt{Q_1^2 - Q_2^2 - Q_3^2}} \left(1 + \sqrt{1 + 8\sqrt{Q_1^2 - Q_2^2 - Q_3^2}}\right)}
\end{align}
The associated Fisher-Souriau information metric $I(\beta) =
\frac{\partial^2 \Phi}{\partial \beta_i \partial \beta_j}$ splits
into the following exact components:
\begin{align}
I(\beta) &=
\frac{1}{\left(\beta^{2}_{1}-\beta^{2}_{2}-\beta^{2}_{3}\right)^{\frac{3}{2}}}\left(
  \begin{array}{ccc}
    -\beta^{2}_{2}-\beta^{2}_{3} & \beta_{1}\beta_{2} & \beta_{1}\beta_{3} \\
    \beta_{1}\beta_{2} & -\beta^{2}_{1}+\beta^{2}_{3} & -\beta_{2}\beta_{3} \\
    \beta_{1}\beta_{3} & -\beta_{2}\beta_{3} & -\beta^{2}_{1}+\beta^{2}_{2}
  \end{array}
\right) \nonumber \\
&\quad +
\frac{1}{\left(\beta^{2}_{1}-\beta^{2}_{2}-\beta^{2}_{3}\right)^{2}}\left(
  \begin{array}{ccc}
    -\beta^{2}_{1}-\beta^{2}_{2}-\beta^{2}_{3} & 2\beta_{1}\beta_{2} & 2\beta_{1}\beta_{3} \\
    2\beta_{1}\beta_{2} & -\beta^{2}_{1}-\beta^{2}_{2}+\beta^{2}_{3} & -2\beta_{2}\beta_{3} \\
    2\beta_{1}\beta_{3} & -2\beta_{2}\beta_{3} & -\beta^{2}_{1}+\beta^{2}_{2}-\beta^{2}_{3}
  \end{array}
\right)
\end{align}
%%%An appendix contains supplementary information that is not an essential part of the text itself but which may be helpful in providing a more comprehensive understanding of the research problem or it is information that is too cumbersome to be included in the body of the paper.
%
%%%=============================================%%
%%% For submissions to Nature Portfolio Journals %%
%%% please use the heading ``Extended Data''.   %%
%%%=============================================%%
%
%%%=============================================================%%
%%% Sample for another appendix section                 %%
%%%=============================================================%%
%
%%% \section{Example of another appendix section}\label{secA2}%
%%% Appendices may be used for helpful, supporting or essential material that would otherwise
%%% clutter, break up or be distracting to the text. Appendices can consist of sections, figures,
%%% tables and equations etc.
%
\end{appendices}

%%===========================================================================================%%
%% If you are submitting to one of the Nature Portfolio journals, using the eJP submission   %%
%% system, please include the references within the manuscript file itself. You may do this  %%
%% by copying the reference list from your .bbl file, paste it into the main manuscript .tex %%
%% file, and delete the associated \verb+\bibliography+ commands.                            %%
%%===========================================================================================%%

\bibliography{bibliography}% common bib file
%% if required, the content of .bbl file can be included here once bbl is generated
%%\input sn-article.bbl

\end{document}